\documentclass[oneside]{amsart}
\usepackage{amssymb}
\usepackage{amsxtra}
\usepackage{amsmath}
\usepackage{amsthm}
\usepackage[all]{xy}

\usepackage[pagebackref]{hyperref}

\usepackage[english]{babel}

\usepackage[notref,notcite]{showkeys}

\DeclareMathOperator{\PGL}{PGL}

\DeclareMathOperator{\HH}{H}

\DeclareMathOperator{\Cent}{Cent}

\DeclareMathOperator{\Rad}{Rad}

\DeclareMathOperator{\et}{\text{\it \'et}}

\DeclareMathOperator{\PGO}{PGO}
\DeclareMathOperator{\PSp}{PGSp}

\renewcommand{\sc}{sc}

\newtheorem{lem}{Lemma}
\newtheorem*{thm*}{Theorem}
\newtheorem{thm}{Theorem}

\newtheorem*{cor*}{Corollary}

\title{Grothendieck---Serre conjecture via embeddings}
\thanks{V. Petrov is partially supported by PIMS fellowship and RFBR~08-01-00756; V. Petrov and
A. Stavrova are supported by RFBR~09-01-00878; I.Panin is supported by the joint DFG--RFBR project
09-01-91333-NNIO-a}
\author{I. Panin}
\author{V. Petrov}
\author{A. Stavrova}%\address{St. Petersburg State University, St. Petersburg, Russia}
\date{20.05.2009}
\begin{document}
%\selectlanguage{russian}

\begin{abstract}
Assume that $R$ is a semi-local regular ring containing an infinite perfect
field, or that $R$ is a semi-local ring of several points on a smooth scheme over an infinite field.
Let $K$ be the field of fractions of $R$.
Let $H$ be a strongly inner adjoint simple algebraic group of type $E_6$ or $E_7$ over $R$,
or any twisted form of one of the split groups of classical type
${\mathrm O}^+_{n,R}$, $n\ge 4$; $\PGO_{n,R}$, $n\ge 4$; $\PSp_{2n,R}$, $n\ge 2$;
$\PGL_{n,R}$, $n\ge 2$.
We prove that the kernel of the map
$$ \HH^1_{\et}(R,H)\to \HH^1_{\et}(K,H) $$
\noindent
induced by the inclusion of $R$ into $K$ is trivial. This continues the recent series of
papers by the authors and N.Vavilov on the Grothendieck---Serre conjecture~\cite[Rem. 1.11]{Gr}.
\end{abstract}

\maketitle

%\section{Introduction}

We prove the following theorem.

\begin{thm}\label{th:main}
Let $R$ be a semi-local domain. Assume moreover that $R$ is regular and contains a infinite perfect
field $k$, or that
$R$ is a semi-local ring of several points on a $k$-smooth scheme over an infinite field $k$. Let $K$ be
the field of fractions of $R$. Let $H$ be a simple group scheme over $R$ that is in the following list:\\
\indent$\bullet$ an adjoint strongly inner simple group of type $E_6$;\\
\indent$\bullet$ an adjoint strongly inner simple group of type $E_7$;\\
\indent$\bullet$ a twisted form of the split group ${\mathrm O}^+_{n,R}$, $n\ge 4$;\\
\indent$\bullet$ a twisted form of the split group $\PGO_{n,R}$, $n\ge 4$;\\
\indent$\bullet$ a twisted form of the split group $\PSp_{2n,R}$, $n\ge 2$;\\
\indent$\bullet$ a twisted form of the split group $\PGL_{n,R}$, $n\ge 2$.\\
Then
the map
$$\HH^1_{\et}(R, H) \to \HH^1_{\et}(K, H)$$
induced by the inclusion of $R$ into $K$ has trivial kernel.
\end{thm}

%\section{Proof of the theorem}

%In what follows we use the notation and terminology of~\cite{PS-tind}. Our numbering of vertices
%of Dynkin diagrams follows~\cite{Bu}.
Note that we follow~\cite{SGA} in abbreviating ``reductive group
scheme'' to ``reductive group''.

Recall that the {\it radical} $\Rad(G)$ of a reductive group $G$ is the unique maximal torus of the
group scheme center $\Cent(G)$ of $G$~\cite[D\'ef. 4.3.6]{SGA}. The quotient $G/\Rad(G)$ is a semisimple
group.

\begin{lem}\label{lem:rad}
Let $R$ be a semi-local regular ring with a field of fractions $K$, and let $G$ be an isotropic reductive group over $R$. Let $P$
be a parabolic subgroup of $G$, $L$ be a Levi subgroup of $P$. If the natural map
$$
\HH^1_{\et}(R,G)\to \HH^1_{\et}(K,G)
$$
has trivial kernel, then the maps
$$
\HH^1_{\et}(R,L)\to \HH^1_{\et}(K,L),\qquad
\HH^1_{\et}(R,L/\Rad(L))\to \HH^1_{\et}(K,L/\Rad(L))
$$
have trivial kernels.
\end{lem}
\begin{proof}
By~\cite[Exp. XXVI Cor. 5.10 (i)]{SGA} the maps
$$
\HH^1_{\et}(R, L)\to\HH^1_{\et}(R, G)\quad\mbox{and}\quad\HH^1_{\et}(K, L)\to\HH^1_{\et}(K, G)
$$
are injective, therefore the first map of the claim has trivial kernel. Set $H=L/\Rad(L)$. The short
exact sequence
$$
1\to \Rad(L)\to L\to H\to 1
$$
leads to the commutative diagram
\begin{equation}\label{eq:diag1}
\xymatrix{
\HH^1_{\et}(R,\Rad(L))\ar[r]\ar[d]&\HH^1_{\et}(R, L)\ar[d]\ar[r]&\HH^1_{\et}(R, H)\ar[d]\ar[r]^\delta&\HH^2_{\et}(R, \Rad(L))\ar[d]\\
\HH^1_{\et}(K,\Rad(L))\ar[r]&\HH^1_{\et}(K, L)\ar[r]&\HH^1_{\et}(K, H)\ar[r]^\delta&\HH^2_{\et}(K, \Rad(L)).
}
\end{equation}
If $\Rad(L)$ is a split torus, then we have
$\HH^1_{\et}(R,\Rad(L))=\HH^1_{\et}(K,\Rad(L))=0$, and
the map $\HH^2_{\et}(R, \Rad(L))\to\HH^2_{\et}(K, \Rad(L))$ is injective (for example,~\cite[Ch. III, Example 2.22]{Milne}).
In general $\Rad(L)$ is a Weil restriction of a split torus defined over a finite \'etale extension of $R$,
therefore, the same statements hold by Shapiro's lemma.
The proof is finished by diagram chasing.
\end{proof}

\begin{lem}\label{lem:embed}
Let $R$ be a semi-local connected ring, and let $H$ be a simple group over $R$ of the same type
as in Theorem~\ref{th:main}. Then there exists an isotropic simple group $G$ over $R$ with
a parabolic subgroup $P$, such that $H$ is isomorphic to a direct factor of the quotient
$L/\Rad(L)$ for a Levi subgroup $L$ of $P$.
\end{lem}
\begin{proof}
Let $H^{\sc}$ be the unique simply connected simple group over $R$ such that $H$ is a central quotient
of $H^{\sc}$. By the main result of~\cite{PS-tind} there exists an isotropic simple simply connected group
$\tilde G$ over $R$ with a parabolic subgroup $\tilde P$ such that $H^{\sc}$ is a normal subgroup
of the semisimple group $[\tilde L,\tilde L]$ (the algebraic
commutator group) for a Levi subgroup $\tilde L$ of $\tilde P$. To be more specific, we note that
if $H$ is a group of strongly inner type $E_6$ resp. $E_7$, then $\tilde G$ is a group of type $E_7$
resp. $E_8$, and $\tilde P$ is a parabolic subgroup of type $P_7$ resp. $P_8$ of $\tilde G$ (with the numbering of the
Dynkin digram as in~\cite{Bu}). If $H^{\sc}$ is a group of type $A_n$, $n\ge 1$, then $\tilde G$ is
an isotropic simply connected group of type $A_l$, $l\ge n+2$, and $\tilde P$ is a parabolic
subgroup of type $P_{d,l+1-d}$, where $d=\frac{l+2-n}2$.
If $H^{\sc}$ is a group of type $B_n$, $n\ge 2$, then $\tilde G$ is
an isotropic simply connected group of type $B_{n+1}$, and $\tilde P$ is a parabolic
subgroup of type $P_1$.
If $H^{\sc}$ is a group of type $C_n$, $n\ge 2$, resp. of type $D_n$, $n\ge 4$, then $\tilde G$ is
an isotropic simply connected group of type $C_l$ resp. $D_l$, where $l\ge n+1$, and $\tilde P$ is a parabolic
subgroup of type $P_d$, where $d=l+1-n$.

Denote by $H_0$, $\tilde G_0$, $\tilde L_0$
etc. the corresponding quasi-split groups over $R$. Computing the root data, one sees that there exists
a central quotient $G_0$ of $\tilde G_0$ such that if $P_0$ and $L_0$ are the images of $\tilde P_0$
and $\tilde L_0$ in $G_0$, then $L_0/\Rad(L_0)$ is isomorphic either to $H_0$, or to a direct product
of $H_0$ by one or two split simple groups of type $A_{d-1}$, $d\ge 2$. Then take $G$ to be the central
quotient of $\tilde G$ that is an inner twisted form of $G_0$. Then, clearly, $H$ is isomorphic to
a direct factor of the quotient $L/\Rad(L)$ of a Levi subgroup of the corresponding parabolic subgroup $P$
of $G$.
\end{proof}

\begin{proof}[Proof of the main theorem]
By Lemma~\ref{lem:embed} $H$ is a direct factor of $L/\Rad(L)$ for a Levi subgroup $L$ of a parabolic
subgroup of an isotropic simple group $G$ over $R$. Since $G$ is isotropic, by the main result of~\cite{PaSV}
(see also~\cite{Pa-newpur}) the map
$$
H_{\et}^1(R,G)\to H_{\et}^1(K,G)
$$
has trivial kernel. Then by Lemma~\ref{lem:rad} the map
$$
H_{\et}^1(R,L/\Rad(L))\to H_{\et}^1(K,L/\Rad(L))
$$
has trivial kernel. Since $H$ is a direct factor of $L/\Rad(L)$, this implies that
the map
$$
H_{\et}^1(R,H)\to H_{\et}^1(K,H)
$$
also has trivial kernel.
\end{proof}

%\noindent{\bf Remark.} The proof is also valid for adjoint (semi)simple algebraic groups of
%classical type.

\bigskip
The authors heartily thank N. Vavilov for inspiring conversations on the subject of this
preprint.

\end{document}